\documentclass{amsart}

\makeatletter
 \def\LaTeX{\leavevmode L\raise.42ex
   \hbox{\kern-.3em\size{\sf@size}{0pt}\selectfont A}\kern-.15em\TeX}
\makeatother

\newcommand{\BibTeX}{{\rm B\kern-.05em{\sc
i\kern-.025emb}\kern-.08em\TeX}}

\newtheorem{thm}{Theorem}[section]
\newtheorem{defn}{Definition}

\theoremstyle{defn}
\newtheorem{lem}[thm]{Lemma}

\newtheorem{remark}[thm]{Remark}

\numberwithin{equation}{section}

\begin{document}

\title[A Weak Weyl's Law]{A Weak Weyl's Law on  compact  metric measure  spaces}

\author{Isaac Z. Pesenson}
\address{Department of Mathematics, Temple University,
Philadelphia, PA 19122} \email{pesenson@temple.edu}

\keywords{      }
  \subjclass{ 43A85, 41A17;}

 \begin{abstract}

The well known  Weyl's Law (Weyl's  asymptotic formula) gives an approximation to the number $\mathcal{N}_{\omega}$  of eigenvalues (counted with multiplicities) on a large interval $[0,\>\omega]$ of the Laplace-Beltrami operator  on a compact Riemannian manifold ${\bf M}$.
In this paper we prove a kind of a weak version of the Weyl's law on certain  compact metric measure spaces ${\bf X}$ which are equipped with a self-adjoint non-negative operator $\mathcal{L}$ acting in $L_{2}({\bf X})$.
 Roughly speaking,  we show that if a certain Poincar\'e inequality holds then $\mathcal{N}_{\omega}$ is controlled by the cardinality of an appropriate cover $\mathcal{B}_{\omega^{-1/2}}=\{B(x_{j},\omega^{-1/2})\},\>\>\>x_{j}\in {\bf X},$ of ${\bf X}$ by balls of radius $\omega^{-1/2}$. 
 Moreover, an opposite inequality holds if the heat kernel that corresponds to $\mathcal{L}$ satisfies short time Gaussian estimates.

It is known that  in the  case of  the so-called  strongly local regular with a complete intrinsic metric Dirichlet spaces the Poincar\'e  holds iff the corresponding heat kernel satisfies short time Gaussian estimates. Thus for such spaces one obtains that $\mathcal{N}_{\omega}$ is essentially equivalent to the cardinality of a cover $\mathcal{B}_{\omega^{-1/2}}$.

\end{abstract}

\maketitle

	\section{Introduction}
	
The well known  Weyl's asymptotic formula gives an approximation to the number $\mathcal{N}_{\omega}$  of eigenvalues (counted with multiplicities) on an interval $[0,\>\omega]$ of the Laplace-Beltrami operator  on a compact Riemannian manifold ${\bf M}$. Namely, according to the  Weyl's asymptotic formula one has for large $\omega$
\begin{equation}\label{Weyl-1}
\mathcal{N}_{\omega}\sim C \>Vol(\mathbf{\mathbf{M}})\omega^{n/2},
\end{equation}
where $n=dim \>\mathbf{\mathbf{M}}$ and $C$ depends on ${\bf M}$. 
Let's rewrite it in  the following form
\begin{equation}\label{Weyl-2}
\mathcal{N}_{\omega}\sim C \>Vol(\mathbf{\mathbf{M}})\omega^{n/2}=C\frac{ Vol({\bf M})}{\left(\omega^{-1/2}\right)^{n}},\>\>\>n=dim\> {\bf M}.
\end{equation}	
Since in the case of a Riemannian manifold ${\bf M}$ of dimension $n$ 
all the balls of the same small radius $\epsilon$ have essentially the same volume $\sim \epsilon^{n}$  
the last fraction  can be  interpreted as a  number of balls $
	B(x_{\nu},\omega^{-1/2})$ with centers $\{x_{\nu}\}$ and radius $\omega^{-1/2}$ in an  "optimal" cover of ${\bf M}$.

Alternatively, if $M_{\omega^{-1/2}}$ is the set of centers of balls in such "optimal" cover of  ${\bf M}$ than the above arguments show that there exist two positive constants $C_{1}, \>C_{2}$ which depend of ${\bf M}$ but independent of $\omega>0$ for which
\begin{equation}\label{WWL-00}
C_{1}\>card\left(M_{\omega^{-1/2}}\right)\leq \mathcal{N}_{\omega}\leq C_{2}\>card\left(M_{\omega^{-1/2}}\right).
\end{equation} 
This double inequality is a statement which is weaker than the asymptotic formula and it is what we call  the \textit{ Weak Weyl's Law}. It should be  noted that  interpretation of (\ref{Weyl-1}) as a statement about a cover requires the assumption that all the balls of the same small radius have essentially the same volume (which is clearly not satisfied in general see section \ref{examples}   below).  However, the weaker statement (\ref{WWL-00}) does not require a such  assumption. 
	
\bigskip

 The main goal of the paper is to sketch a direct proof of  inequalities similar to (\ref{WWL-00}) inequality (see Theorem \ref{WWL}) in a setting of  compact metic measure spaces. In the case of compact Riemannian manifolds a similar statement was proved in \cite{Pes19}.	
		
	\bigskip
	
	It should be noted that a complete analog of the classical Weyl's asymptotic formula was recently established in \cite{AHT}, \cite{Z} under very strong assumptions about the curvature of the underlying metric measure space. We prove only a weaker form of the Weyl's asymptotic formula but our assumptions about a space  seems to be much weaker.
	Generally speaking, we 
 will work  in a framework in which the  main objects are:

	(a) a compact metric measure space ${\bf X}$ with  doubling property,

	(b) a non-negative densely defined symmetric  operator $\mathcal{L}$ in the corresponding space $L_{2}({\bf X})$ whose closure is self-adjoint.

	First we show that
there exists a constant $N=N(\mathbf{X})$ and for every sufficiently small $\rho>0$ there exists a set of points $\mathcal{X}_{\rho}=\{x_{j}\},\>x_{j}\in {\bf \mathbf{X}}$ such that
\begin{enumerate}

	\item  Balls $B(x_{j}, \rho/2)$ are disjoint;
	
		\item $\left\{B(x_{j}, \rho)\right\}$ is a cover of $\mathbf{X}$ and its  multiplicity  is not greater than $\leq N$.

\end{enumerate}
	
	We say that that $\mathcal{X}_{\rho}=\{x_{j}\},\>x_{j}\in {\bf \mathbf{X}},\>\rho>0,$  is a metric ($\rho$, N)-lattice if it  satisfies properties (1) and (2).

	In our first main Theorem \ref{lower-estimate} we prove that modulo some techical assumptions \textit{if the Poincar\'e inequality holds in the form   (\ref{Poinc-0}) (see  below)} then  the operator $\mathcal{L}$ has  a discrete spectrum $0\leq\lambda_{0}\leq \lambda_{1}\leq \lambda_{2},...$ which goes to infinity  without any accumulation points. Thus it makes sense to introduce  $\mathcal{N}_{\omega}(\mathcal{L})$  as  a number   of eigenvalues of $\mathcal{L}$ (counted with multiplicities) which are not greater than $\omega>0$. The last part of the Theorem    \ref{lower-estimate}        says that under the same assumptions as in the first part of the Theorem,  there exists a constant 
		$0<\gamma=\gamma({\bf  X})<1$
		 such that for all sufficiently large $\omega>0$ the following inequality  holds
		\begin{equation}\label{Upper}
		\mathcal{N}_{\omega}(\mathcal{L}) \leq \> \inf \>card \left(\mathcal{X}_{\gamma\omega^{-1/2}}\right),
		\end{equation}
		where  $\inf$ is taken over all $(\gamma\omega^{-1/2}, N)$-lattices.

	  The second result is established    under  different assumptions. Namely, in Theorem \ref{WWL-left} we show that \textit{if the heat kernel of the heat semigroup generated by $\mathcal{L}$ satisfies short time Gaussian estimates (see (\ref{Gauss})}) then  there exists   a constant
		$$
		0<c=c({\bf X})<1
		$$
		 such that for all sufficiently large $\omega$ the number $\mathcal{N}_{\omega}(\mathcal{L})$ of eigenvalues of $\mathcal{L}$ in $[0,\>\omega]$ satisfies the following  inequality 
		\begin{equation}\label{Main-left}
	c\>\sup \>card \left(\mathcal{X}_{\omega^{-1/2}}\right)\leq \mathcal{N}_{\omega}(\mathcal{L}),
		\end{equation}
		where  $\sup$ is taken over all $(\omega^{-1/2}, N)$-lattices.  It should be nouted that the proof of this  lower estimate heavily depend on some ideas in \cite{CKP} and, particular, on the first part of their Lemma  3.19 (see Lemma \ref{key} in the present paper). Our section \ref{proof} contains  a slightly modified proof of this lemma.

	   We combine these two statements in Theorem \ref{WWL}.
	In section
\ref{Dirichlet} we remind  a known fact that under some additional assumptions on ${\bf X}$ and $\mathcal{L}$ (i.e. in the setting of a special type of metric measure spaces  the so-called Dirichlet spaces \cite{   A}, \cite{BH},  \cite{      FOT}, \cite{S1}-\cite{ S3}) \textit{existence of the Poincar\'e inequality is equivalent to the "correct" short time Gaussian estimates}(\cite{HS}, \cite{S3}). 	After that in section \ref{examples} we illustrate   our results in the case of   sub-Riemannian   compact homogeneous manifolds.

\section{Metric measure spaces and Poincare inequality}

A compact metric measure space in the sense of Coifman and Weiss \cite{CW1}, \cite{CW2}, is a triple $(\mathbf{X}, d, \mu)$  where $\mathbf{X}$ is a compact metric space  with a metric $d$ and a positive Radon  measure $\mu$ such that the  doubling condition  holds.
Namely,

\bigskip

{\bf Assumption A}.   \textit{There exists a  $D>0,\>$ such that for every open ball $B(x,r)$ with center $x\in \mathbf{X}$ and radius $r>0$, the following inequality holds
\begin{equation}\label{doubling}
0<\left|B(x, 2r)\right|\leq 2^{D}\left|B(x, r)\right|,\>\>\>  
\end{equation}
where $|\cdot| $ is the volume of a ball.}

Our next  assumptions are about "analysis". 

\bigskip

{\bf Assumption B}. \textit{The real space $L_{2}(\mathbf{X}, \mu)$ is equipped with a non-negative symmetric operator  $\mathcal{L}$ defined on a dense domain $\mathcal{D}(\mathcal{L})$:
$$
\left<\mathcal{L}f, g\right>=\left<f, \mathcal{L}g\right>,\>\>\>f,g\in \mathcal{D}(\mathcal{L}).
$$
 We assume that  $\mathcal{D}^{\infty}(\mathcal{L})=\bigcap_{k\in \mathbb{N}} \mathcal{D}(\mathcal{L}^{k})$ is a subalgebra of the space of continuous functions $C({\bf X})$.
}
\bigskip

These assumptions along with compactness of $\bf{X}$ imply that  the following bilinear map is well   defined on $\mathcal{D}^{\infty}(\mathcal{L})\times \mathcal{D}^{\infty}(\mathcal{L})$ 
and maps it into $ L_{1}(\bf{X},\mu)$ 
$$
\Gamma(f,g)=\frac{1}{2}\left(\mathcal{L}(fg)-f\mathcal{L}g-g\mathcal{L}f\right),\>\>\>f,g\in \mathcal{D}^{\infty}(\mathcal{L}).
$$
The following function is a sort of "squared gradient"  of  $f\in   \mathcal{D}^{\infty}(\mathcal{L})$:
$$
\Gamma(f,f)=\frac{1}{2}\left(\mathcal{L}(f^{2})-2f\mathcal{L}f\right)\in L_{1}(\bf{X},\mu).
$$
At the same time, by applying the spectral theorem to the operator $\mathcal{L}$ one can consider the positive square root $\mathcal{L}^{1/2}$.

\bigskip

{\bf Assumption C}. \textit{
There exists a $\>\>\>c>0\>\>\>$ such that for  all $f\in \mathcal{D}^{\infty}(\mathcal{L})$ the following inequality holds 
\begin{equation}\label{grad}
\int_{\mathbf{X}}\Gamma(f,f)(x)d\mu(x)\leq c\left\| \mathcal{L}^{1/2}f\right\|^{2}.
\end{equation}   }

\bigskip

Our main assumption which connects the operator $\mathcal{L}$ and the metric measure space is the Poincar\'e inequality. Namely, we assume the following.

\bigskip

{\bf Assumption D}. \textit{
There exists  constant $C>0$ such that for any $f\in \mathcal{D}^{\infty}(\mathcal{L})=\bigcap_{k\in \mathbb{N}}\mathcal{D}(\mathcal{L}^{k}) $ and any  ball $B(x, \rho),\>\>x\in \mathbf{X},$ of a sufficiently small radius $\rho$ the following {\bf local Poincar\'e inequality } holds
\begin{equation}\label{Poinc-0}
\int_{B(x,\rho)}|f(y)-f_{B}|^{2}d\mu(y)\leq C\rho^{2}\int_{B(x,\rho)}\Gamma(f,f)(y)d\mu(y),
\end{equation}
where 
$$
f_{B}=\frac{1}{|B(x,\rho)|}\int_{B(x,\rho)}f(x)d\mu(x).
$$  }

\begin{remark}
It should be stressed  that for our purposes it is sufficient to request that  the Poincar\'e inequality holds only for functions in $ \mathcal{D}^{\infty}(\mathcal{L})$. It is a little less than the common assumption (see for example \cite{CKP}).
\end{remark}

\section{Paley-Wiener functions}

 Following \cite{Pes88}  -\cite{Pes01} we introduce the subspaces     ${\bf E}_{\omega}(\mathcal{ L}),\>\>\omega>0,$ of  Paley-Wiener  functions in $L_{2}({\bf X})$.
We are going to apply the functional model of 
  the spectral theory \cite{BS} to the non-negative self-adjoint operator $\mathcal{L}$. According to this model there exist a
direct integral of Hilbert spaces $A=\int_{0}^{\infty} A(\lambda )dm (\lambda )$ and a
unitary operator $\mathcal{F}$ from
$L_{2}({\bf X}, \mu)$ onto $A$, which transforms domain of $\mathcal{L}^{k}$ onto
$A_{k}=\{a\in A|\lambda ^{k}a\in A \}$
with norm
$$
\|a(\lambda )\|_{A_{k}}= \left (\int^{\infty}_{0} \lambda ^{2k}
\|a(\lambda )\|^{2}_{A(\lambda )} dm(\lambda ) \right )^{1/2} 
$$
besides $\mathcal{F}(\mathcal{L}^{k} f)=\lambda ^{k} (\mathcal{F}f), $ if
$f$ belongs to the domain of
$\mathcal{L}^{k}$.  As known, $A$ is the set of all $m $-measurable functions
$\lambda \rightarrow a(\lambda )\in
A(\lambda ) $, for which the norm

$$\|a\|_{A}=\left ( \int ^{\infty }_{0}\|a(\lambda )\|^{2}_{A(\lambda )}
dm (\lambda ) \right)^{1/2} $$
is finite.
\begin{defn}
The space of $\omega$-Paley-Wiener  functions $\mathbf{E}_{\omega}(\mathcal{L}),\>\>\omega>0,$ is defined as the set of all functions whose "Fourier transform" $\mathcal{F}f$ has
support in
 $[0 , \omega ] $. 
 \end{defn}

Note that at this point we don't  have any information about  the spectrum of $\mathcal{L}$.

\begin{lem}
The space  $\mathbf{E}_{\omega}(\mathcal{L})$ consists of  all $f\in \mathcal{D}^{\infty}(\mathcal{L})=\bigcap_{j\in \mathbb{N}}\mathcal{D}(\mathcal{L}^{j}) $ for which the following Bernstein inequalities hold
\begin{equation}\label{Bern}
\left \|\mathcal{L}^{k/2}f\right\|\leq \omega^{k}\|f\|,\>\>k\in \mathbb{N}.
\end{equation}

\end{lem}\section{The upper estimate}

First, we describe what will be called a metric lattice. 
\begin{lem}
There exists a constant $N=N(\mathbf{X})$ and for every sufficiently small $\rho>0$ there exists a set of points $\mathcal{X}_{\rho}=\{x_{j}\},\>x_{j}\in {\bf \mathbf{X}}$ such that
\begin{enumerate}

	\item  Balls $B(x_{j}, \rho/2)$ are disjoint;
	
		\item $\left\{B(x_{j}, \rho)\right\}$ is a cover of $\mathbf{X}$ and its  multiplicity  is not greater than $\leq N$.

\end{enumerate}
\end{lem}

\begin{proof}
Let us choose a family of disjoint balls $B(x_{i},\rho/2)$ such
that there is no ball $B(x,\rho/2), x\in M,$ which has empty
intersections with all balls from our family. Then the family
$B(x_{i},\rho)$ is a cover of $M$. Every ball from the family
$\{B(x_{i}, \rho)\}$, that has non-empty intersection with a
particular ball $B(x_{j}, \rho)$ is contained in the ball
$B(x_{j}, 2\rho)$. Since any two balls from the family
$\{B(x_{i},\rho/2)\}$ are disjoint, it gives the following
estimate for the index of multiplicity $N$ of the cover
$\{B(x_{i},\rho)\}$:
\begin{equation}
N\leq \sup_{y\in \mathbf{X}}\frac{|B(y, 2\rho)|}{\inf_{x\in B(y, 2\rho)|}|B(x, \rho/2)|}.
\end{equation}
Note, that the doubling inequality 
$$
0<\left|B(x, 2r)\right|\leq 2^{D}\left|B(x, r)\right|
$$
implies 
 the next two inequalities 
\begin{equation}\label{1a}
\left|B(x, \rho r)\right|\leq (2\rho)^{D}|B(x, r)|,\>\>\>x\in \mathbf{X}, \>\>r>0, \>\>\rho>1, 
\end{equation}
and
\begin{equation}\label{1b}
\left|B(x, r)\right|\leq 2^{D}\left(1+d(x,y)/r\right)^{D}|B(y, r)|,\>\>\>x,y\in \mathbf{X}, \>\>r>0.
\end{equation}
Indeed, if $2^{k}\leq \rho<2^{k+1}$ then $2^{(k+1)D}\leq (2\rho)^{D}$ and 
$$
\left|B(x, \rho r)\right|\leq \left|B(x, 2^{k+1}r)\right|\leq 2^{(k+1)D}\left|B(x, r)\right|\leq (2\rho)^{D}\left|B(x, r)\right|.
$$
On the other hand, one has $B(x,r)\subset B\left(y, r+d(x,y)\right)=B\left(y, \left(1+d(x,y)/r\right)r\right)$ and it  implies (\ref{1b}).

Now, since $x\in B(y, 2\rho)$ we have $d(x,y)\leq 2\rho$ and according to (\ref{1b}) 
$$
|B(y, \rho/2)|\leq 2^{D}\left( 1+\frac{d(x,y)}{\rho/2}\right)^{D} |B(x, \rho/2)|\leq 10^{D} |B(x, \rho/2)|.
$$
In other words for any $y\in \mathbf{X}$ and every $x\in B(y, 2\rho)$  one has
\begin{equation}\label{reverse}
|B(x, \rho/2)|\geq 10^{-D}|B(y, \rho/2)|.
\end{equation}
This inequality along with (\ref{1a}) allows to continue estimation
of $N$:

$$N\leq  10^{D} 
\sup_{y\in \mathbf{X}}\frac{|B(y,2\rho)|}{|B(x,\rho/2)|} \leq
10^{D}\sup_{y\in \mathbf{X}}\frac{8^{D}|B(y,\rho/2)|}{|B(y, \rho/2)|}=80^{D}=N(\mathbf{X}).
$$

\end{proof}

\begin{defn}
Every set  that $\mathcal{X}_{\rho}=\{x_{j}\},\>x_{j}\in {\bf \mathbf{X}},\>\rho>0,$ that satisfies properties of the previous lemma will be called a  metric ($\rho$, N)-lattice.
\end{defn}

We are going to prove the following theorem.

\begin{thm} \label{lower-estimate}Assume that for a compact measure metric space ${\bf X}$ and operator $\mathcal{L}$ all the above {\bf Assumptions   A-D} are satisfied. 
Then 
\begin{enumerate}

\item Every space $ {\bf E}_{\omega}(\mathcal{L})$ is finite-dimensional.
\item The operator $\mathcal{L}$ has  a discrete spectrum $0\leq\lambda_{0}\leq \lambda_{1}\leq \lambda_{2},...$ which goes to infinity  without any accumulation points  and there exists a complete  family  $\{\psi_{j}\}$  of orthonormal eigenfunctions which form a  basis in $L_{2}(\mathbf{X}, \mu)$.
\item There exists constant 
		$0<\gamma=\gamma({\bf  X})<1$
		 such that for all sufficiently large $\omega$ the following inequality  holds
		\begin{equation}\label{Upper}
		\mathcal{N}_{\omega} \leq \> \inf \>card \left(\mathcal{X}_{\gamma\omega^{-1/2}}\right),
		\end{equation}
		where   $\mathcal{N}_{\omega}$ is a number   of eigenvalues of $\mathcal{L}$ (counted with multiplicities) which are not greater than $\omega>0$, and $\inf$ is taken over all $(\gamma\omega^{-1/2}, N)$-lattices.
		\end{enumerate}
\end{thm}
	
\begin{proof}
We note that  for all $\alpha>0$
$$
|A|^{2}\leq \left(|A-B|+|B|\right)^{2}\leq |A-B|^{2}+2|A-B||B|+|B|^{2},
$$
and
$$
2|A-B||B|\leq\alpha^{-1}|A-B|^{2}+\alpha|B|^{2},\>\>\>\>\alpha>0,
$$
which imply  the inequality 
\begin{equation}\label{in}
(1+\alpha)^{-1}|A|^{2}\leq\alpha^{-1}|A-B|^{2}+|B|^{2},\>\>\alpha>0.
\end{equation}

Let $\mathcal{X}_{\rho}=\{x_{j}\}$ be a $(\rho, N) $-lattice and 
$$ 
\zeta_{j}(x)=|B(x_{j}, \rho)|^{-1}\chi_{j}(x),\>\>\>\>
 \xi_{j}(x)=|B(x_{j}, \rho)|^{-1/2}\chi_{j}(x),
$$ 
where $\chi_{j}$ is characteristic function of $B(x_{j}, \rho)$.
 One has for any $f\in L_{2}(\bf{X},\mu)$ and any $\alpha>0$,
\begin{equation}\label{firstineq}
\|f\|^{2}\leq \sum_{j}\int_{B(x_{j}, \rho)}|f|^{2}\leq
$$
$$\leq (1+\alpha)\sum_{j}\int_{B(x_{j}, \rho)}|f-     \left<f, \zeta_{j}\right>|^{2}+\frac{1+\alpha}{\alpha}\sum_{j}| \left<f, \xi_{j}\right>|^{2}.
\end{equation}
According to Poincar\'e inequality (\ref{Poinc-0}) if $f$ belongs to $\mathcal{D}^{\infty}(\mathcal{L})$ then
$$
\sum_{j}\int_{B(x_{j},\rho) }|f-          \left<f, \zeta_{j}\right>    |^{2}\leq C\rho^{2}\sum_{j}\int_{B(x_{j},\rho)}\Gamma(f,f)(y)d\mu(y).
$$
Since cover by $B(x_{j}, \rho)$ has finite multiplicity $N$ it gives
$$
\sum_{j}\int_{B(x_{j},\rho) }|f-          \left<f, \zeta_{j}\right>    |^{2}\leq CN\rho^{2}\int_{{\bf X}}\Gamma(f,f)(y)d\mu(y)
$$
and thanks to the gradient property
$$
\int_{\mathbf{X}}\Gamma(f,f)(x)d\mu(x)\leq c\left\| \mathcal{L}^{1/2}f\right\|^{2}.
$$
 it implies
\begin{equation}\label{2ineq}
\sum_{j}\int_{B(x_{j},\rho) }|f-          \left<f, \zeta_{j}\right>    |^{2}\leq CcN\rho^{2}\left\| \mathcal{L}^{1/2}f\right\|^{2}.
\end{equation}
Thus combining (\ref{firstineq}) and (\ref{2ineq}) we obtain a kind of global Poincar\'e inequality:

\begin{equation}
\|f\|^{2}\leq \frac{1+\alpha}{\alpha}\sum_{j}|\left<f,\xi_{j}\right>|^{2}+(1+\alpha)CcN\rho^{2}\left\| \mathcal{L}^{1/2}f\right\|^{2}.
\end{equation}
If $f\in {\bf E}_{\omega}(\mathcal{L})$ then the Bernstein inequality gives 
$$
\|f\|^{2}\leq \frac{1+\alpha}{\alpha}\sum_{j}|\left<f,\xi_{j}\right>|^{2}+ (1+\alpha)CcN(\rho\omega^{1/2})^{2}\|f\|^{2}.
$$ 
If $\gamma<\left((1+\alpha)CcN\right)^{-1}$ then  by choosing $\rho$ which satisfies 
$$
\rho=\gamma\omega^{-1/2}
$$
we can move the second term on the right  to the left side to obtain for $f\in {\bf E}_{\omega}(\mathcal{L})$
$$
\|f\|^{2}\leq \frac{1}{1-\tau}\frac{1+\alpha}{\alpha}\sum_{j}| \left<f, \xi_{j}\right>|^{2}=C_{1}\sum_{j}| \left<f, \xi_{j}\right>|^{2},
$$
where
$$
\tau=(1+\alpha)CcN\gamma^{2}<1,\>\>\>\>C_{1}=\frac{1}{1-\tau}\frac{1+\alpha}{\alpha},\>\>\>\>C_{1}=C_{1}(C,c,N,\alpha     ).
$$
Since  the space ${\bf X}$  is compact  every lattice contains only a  finite number of points. The last inequality shows that every function in ${\bf E}_{\omega}(\mathcal{L})$ is determined by a finite number of functionals.
It gives the first two items of our statement.

Moreover, the  dimension $\mathcal{N}_{\omega}$ of the space $ {\bf E}_{\omega}(\mathcal{L})$  is exactly the same as a number   of eigenvalues of $\mathcal{L}$ (counted with multiplicities) which are not greater than $\omega>0$.         Since it is clear that  the  dimension $\mathcal{N}_{\omega}$ cannot be bigger than cardinality of a sampling set for $ {\bf E}_{\omega}(\mathcal{L})$,      we obtain the third item of our theorem. Theorem is proved.

\end{proof}

\section{The lower estimate}

We are going to make use of the heat kernel 
$
P_{t}(x,y)
$ 
which is associated with the heat semigroup $e^{-t\mathcal{L}}$ generated by the self-adjoint operator $\mathcal{L}$:
\begin{equation}\label{heatkernel}
e^{-t\mathcal{L}}f(x)=\int_{{\bf X}}P_{t}(x,y)f(y)\mu(y).
\end{equation}
We assume that in addition to all the previous {\bf Assumptions  A-D} 
the heat kernel satisfies short-time Gaussian estimates

\bigskip

{\bf Assumption E.}\label{Gauss} There exist positive constants $C_{1}, C_{2}, c_{1}, c_{2},$ such that

\textit{
\begin{equation}\label{Gauss}
 \frac{C_{1}e^{-c_{1}\frac{ \left(dist(x,y)\right)^{2}}{t}}}{\sqrt{\mu(B(x, \sqrt{t}))\mu (B(y, \sqrt{t}))}}\leq P_{t}(x,y)\leq  \frac{C_{2}e^{-c_{2}\frac{ \left(dist(x,y)\right)^{2}}{t}}}{\sqrt{\mu(B(x, \sqrt{t}))\mu (B(y, \sqrt{t}))}}
 \end{equation}
where $0<t<1$ and $\mu (B(y, \sqrt{t}))$ is the volume of the ball.}

\bigskip

\begin{remark}
The estimate (\ref{Gauss}) "knows" that in the case of a general metric measure space balls of the same radius may have different volumes.

\end{remark}

Our objective in this section is to prove the following statement.
		\begin{thm} \label{WWL-left}  Suppose that a compact metric measure space ${\bf X}$ and  a self-adjoint operator $\mathcal{L}$ in $L_{2}({\bf X})$  and all the {\bf Assumptions A-E} are satisfied.  There exists   a constant
		$$
		0<c=c({\bf X})<1
		$$
		 such that for all sufficiently large $\omega$ the number $\mathcal{N}_{\omega}(\mathcal{L})$ of eigenvalues of $\mathcal{L}$ in $[0,\>\omega]$ satisfies the following  inequality 
		\begin{equation}\label{Main-left}
	c\>\sup \>card \left(\mathcal{X}_{\omega^{-1/2}}\right)\leq \mathcal{N}_{\omega}(\mathcal{L}),
		\end{equation}
		where  $\sup$ is taken over all $(\omega^{-1/2}, N)$-lattices.
	\end{thm}
	
	\begin{proof}
	
The key ingredient in the proof  of Theorem \ref{WWL-left} is the following lemma which contains certain information about pointwise  growth of eigenfunctions $\psi_{l}$ in terms of the metric and measure.

\begin{lem}\label{key}
(CPK, Lemma 3.19)
Under {\bf Assumptions A-E} there exist   constants $ a_{1}=  a_{1}({\bf X}) >0,\>\>\> a_{2}=  a_{2}({\bf X}) >0$ such that for all sufficiently large $s>0$ 

\begin{equation}\label{double}
\frac{a_{1}}{|B(x, s^{-1})|}\leq \sum_{l,\>\lambda_{l}\leq s}|\psi_{l}(x)|^{2}\leq \frac{a_{2}}{|B(x, s^{-1})|}.
\end{equation}
\end{lem}
A proof of this lemma is given in Appendix \ref{proof}. We apply Lemma \ref{key}  when  $t =\omega$ to  obtain the following inequality for sufficiently large  $\omega$: 
\begin{equation}\label{ball}
\frac{1}{|B(x, \omega^{-1/2})|}\leq c_{1}\sum_{l,\>\lambda_{l}\leq \omega}|\psi_{l}(x)|^{2}.
\end{equation}
 One has
 $$
card\left(\mathcal{X}_{\omega^{-1/2}}\right)=\sum_{x_{j}\in \mathcal{X}_{\omega^{-1/2}}}\frac{|B(x_{j}, \omega^{-1/2})|}{|B(x_{j}, \omega^{-1/2})|}=
$$
$$
\sum_{x_{j}\in \mathcal{X}_{\omega^{-1/2}}}\frac{1}{|B(x_{j}, \omega^{-1/2})|}
\int_{B(x_{j}, \omega^{-1/2})}\mu(x).
$$
On the other hand, since for  every $x\in  B(x_{j}, \omega^{-1/2})$ one has $dist(x_{j}, x)\leq \omega^{-1/2}$  the inequality  
$$
\left|B(x, r)\right|\leq 2^{D}\left(1+d(x,y)/r\right)^{D}|B(y, r)|,\>\>\>x,y\in \mathbf{X}, \>\>r>0.
$$
 gives 
$$
|B(x, \omega^{-1/2})|\leq 4^{D}|B(x_{j}, \omega^{-1/2})|.
$$
Thus one has 
$$
\int_{B(x_{j}, \omega^{-1/2})}\frac{1}{|B(x_{j}, \omega^{-1/2})|}\mu(x)\leq 4^{D}\int_{B(x_{j}, \omega^{-1/2})}\frac{\mu(x)}{|B(x, \omega^{-1/2})|}.
$$
This inequality along with (\ref{ball}) shows  that for every sufficiently large $\omega>0$ and every $(\omega^{-1/2}, N)$-lattice $\mathcal{X}_{\omega^{-1/2}}$ the following inequalities hold true
$$
card\left(\mathcal{X}_{\omega^{-1/2}}\right)\leq 4^{D}\sum_{x_{j}\in \mathcal{X}_{\omega^{-1/2}}}\int_{B(x_{j}, \omega^{-1/2})}\frac{\mu(x)}{|B(x, \omega^{-1/2})|}\leq 
$$
$$
4^{D}N\int_{{\bf X}}\frac{\mu(x)}{|B(x, \omega^{-1/2})|}\leq 4^{D}c_{1}\int_{{\bf X}}\left(\sum_{l,\>\lambda_{l}\leq\omega}|\psi_{l}(x)|^{2}\right)\mu(x).
$$
Since
\begin{equation}\label{N}
\int_{{\bf X}}\left(\sum_{l,\>\lambda_{l}\leq \omega}|\psi_{l}(x)|^{2}\right)\mu(x)=\sum_{l,\>\lambda_{l}\leq  \omega}\int_{{\bf X}}|\psi_{l}(x)|^{2}\mu(x)=\mathcal{N}_{ \omega}(\mathcal{L}),
\end{equation}
we obtain  the inequality 
\begin{equation}\label{upper}
card\left(\mathcal{X}_{\omega^{-1/2}}\right)\leq c_{2}\mathcal{N}_{ \omega}(\mathcal{L}),
\end{equation}
for some $c_{2}=c_{2}\left(\mathbf{X}\right)$.
Thus
there exists a positive  $c=c({\bf X})< 1$ such that 
\begin{equation}\label{lower}
c\>\sup card\left(\mathcal{X}_{\omega^{-1/2}}\right)\leq \mathcal{N}_{\omega}(\mathcal{L}),
\end{equation}
where $\sup$ is taken over all  lattices   $\mathcal{X}_{\omega^{-1/2}}$.
Theorem \ref{WWL-left} is proven.

\end{proof}
To summarize we formulate next theorem.

		\begin{thm} \label{WWL}  Assume that for a compact measure metric space ${\bf X}$ and operator $\mathcal{L}$  all the above  {\bf assumptions  A-E}  are satisfied. Then 
		there are   constants $0<c=c({\bf X})<1$ and 
		$0<\gamma=\gamma({\bf  X})<1$
		 such that for all sufficiently large $\omega$ the following double inequality  holds
		$$
		c \>\sup card\left(\mathcal{X}_{\omega^{-1/2}}\right)
		\leq \mathcal{N}_{\omega}(\mathcal{L}) \leq \> \inf card\left(\mathcal{X}_{\gamma\omega^{-1/2}}\right),
		$$
		where $\>\sup\>$ and $\>\inf\>$ are taken over all $(\omega^{-1/2}, N)$-lattices and $(\gamma\omega^{-1/2}, N)$-lattices respectively.
	\end{thm}

\section{Strongly local regular with a complete intrinsic metric Dirichlet spaces}\label{Dirichlet}
	
It is a natural question to ask if there exist conditions under which our  { \bf Assumptions A-E} hold true. It is a known fact \cite{HS, S3} that for the so-called \textit{strongly local regular with a complete intrinsic metric Dirichlet spaces} (see \cite{   A}, \cite{BH},  \cite{      FOT}, \cite{S1}-\cite{ S3})  the Poincar\'e inequality (\ref{Poinc-0})  holds if and only the Gaussian estimates (\ref{Gauss}) hold. In what follows we remind definition of such Dirichlet spaces.
	
Suppose that  we are given a  metric measure space ${\bf X}$ with doubling property (\ref{doubling}),
	and a positive densely defined symmetric operator $\mathcal{L}$ in the corresponding space $L_{2}({\bf X})$

 \subsection{Dirichlet spaces}
 
 We consider a symmetric non-negative form
	$$
	\mathcal{E}(f,g)=\langle\mathcal{L}f,g\rangle=\mathcal{E}(g,f),\>\>\>\>\>\>\mathcal{E}(f, f)=\langle\mathcal{L} f,f\rangle\geq 0,
	$$
	with domain $\mathcal{D}(\mathcal{E})=\mathcal{D}(\mathcal{L})$.
On $\mathcal{D}(\mathcal{E})$ we introduce  the pre-hilbertian closable structure defined as
$$
\|f\|^{2}_{\mathcal{E}}=\|f\|^{2}+\mathcal{E}(f,f)
$$
which has a closure $\overline {\mathcal{E}}$   in  $L_{2}(\mathbf{X}, \mu)$ with domain $\mathcal{D}(\overline{\mathcal{E}})$ .  Now one can construct a self-adjoint extension $\overline{\mathcal{L}}$ of the operator $\mathcal{L}$ whose domain consists of all $f\in \mathcal{D}(\overline{\mathcal{E}})$ for which there exists $v\in L_{2}(\mathbf{X}, \mu)$ such that $\overline{\mathcal{E}}(f,g)=\left<v, g\right>$ for all $g\in \mathcal(\overline{\mathcal{E}})$ and by definition $\overline{\mathcal{L}}f=v$. The operator  $\overline{\mathcal{L}}$ is non-negative, self-adjoint, and 
$$
\mathcal{D}(\overline{\mathcal{E}})=\mathcal{D}\left((\overline{\mathcal{L}})^{1/2}\right),\>\>\>\>\>\>\>\overline{\mathcal{E}}(f,g)=\left<(\overline{\mathcal{L}})^{1/2}f, (\overline{\mathcal{L}})^{1/2}g\right>.
$$
Using the regular spectral theory of positive self-adjoint operators, one can associate with $\mathcal{L}$  a self-adjoint strongly continuous contraction semigroup in $L_{2}(\mathbf{X})$
$$
T(t)=e^{-t\mathcal{L}}=\int_{0}^{\infty}e^{-\sigma t}dE_{\sigma},
$$
where $E_{\sigma}$ is the spectral resolution associated with  $\overline{\mathcal{L}}$.

Our next assumption is that $T(t)$ is a submarkovian semigroup which means that if $0\leq f\leq 1$  and $f\in L_{2}(\mathbf{X})$ then $0\leq T(t)f\leq 1$.

If the above conditions  are satisfied, one says that  the pair $(\mathcal{D}(\overline{\mathcal{E}}), \overline{\mathcal{E}})$ is a \textit{Dirichlet space}.

\subsection {Definition of strongly local regular with a complete intrinsic metric Dirichlet spaces}.

The form $\overline{\mathcal{E}}$ is \textit{strongly local} if 
 $\overline{\mathcal{E}}(f,g)=0,\>\>f, g\in \mathcal{D}(\overline{\mathcal{E}})$  for every $f$ with compact support and $g$ constant on a neighborhood of the support of $f$. 
 The form $\overline{\mathcal{E}}$ is \textit  {regular} if 
the space $C_{0}(\mathbf{X})$ of
continuous functions on $\mathbf{X}$  with compact support has the property that the algebra
  $C_{0}(\mathbf{X}) \cap  \mathcal{D}(\overline{\mathcal{E}}) $ is dense in   $C_{0}(\mathbf{X})$  with respect to the $\sup$  norm, and dense in $\mathcal{D}(\overline{\mathcal{E}}) $  
in the norm 

  $$
  \sqrt{\overline{\mathcal{E}}(f,f)+\|f\|^{2}_{2}}.
  $$
Under the above assumptions  there exists a bilinear symmetric form $d\Gamma$  defined
on $\mathcal{D}(\overline{\mathcal{E}})\times \mathcal{D}(\overline{\mathcal{E}})$  with values in the signed Radon measures on $\mathbf{X}$ such that for $f,g,h\in C_{0}(\mathbf{X})\cap \mathcal{D}(\overline{\mathcal{E}})$
$$
\mathcal{E}(hf,g)+\mathcal{E}(f,hg)-\mathcal{E}(h,fg)=2\int_{\mathbf{X}}h\>d\Gamma(f,g).
$$

In particular, 
$$
\overline{\mathcal{E}}(f,g)=\int_{\mathbf{X}}d\Gamma(f,g),\>\>\>\>d\Gamma(f,f)\geq 0.
$$
If one assums
that $ \mathcal{D}(\mathcal{L})$ is a subalgebra of $C_{0}(\mathbf{X})$,
then $d\Gamma$ is absolutely continuous with respect to original measure $\mu$, i.e. 
$$
d\Gamma(f,g)=\Gamma(f,g)d\mu,
$$
where
$$
\Gamma(f,g)=\frac{1}{2}\left( \mathcal{L}(fg)-f\mathcal{L}g-g\mathcal{L}f\right), \>\>\>f,g\in \mathcal{D}(\mathcal{L}).
$$

To summarize, there exists a bilinear function $\Gamma$ which maps 
$\mathcal{D}(\overline{\mathcal{E}})\times \mathcal{D}(\overline{\mathcal{E}})$ to $L_{1}(\mathbf{X})$  such that 
$\Gamma(f,f)(x)\geq 0$,

$$
\overline{\mathcal{E}}(hf,g)+\overline{\mathcal{E}}(f,hg)-\overline{\mathcal{E}}(h,fg)=2\int_{\mathbf{X}}h\>d\Gamma(f,g),
$$
where $f,g,h \in  \mathcal{D}(\overline{\mathcal{E}}) \cap L_{\infty}(\mathbf{X})$ and

$$
\overline{\mathcal{E}}(f,g)=2\int_{\mathbf{X}} \Gamma(f,g)(x)d\mu(x).
$$
The \textit{intrinsic distance} $\rho(x,y)$  on $\mathbf{X}$ is defined as 
$$
\rho(x,y)=\sup \{ |\phi(x)-\phi(y)| \},
$$
where $\sup$ is taken over all $\phi \in C_{0}(\mathcal{X})\cap \mathcal{D}(\overline{\mathcal{E}})$ for which the "gradient" $\Gamma(\phi, \phi)(x)$ is not greater than one.
Our next important assumption is that 
 \textit{topology generated by the intrinsic distance $\rho(x,y)$  is the same as the original topology of $\mathbf{X}$ and the metric space $(\mathbf{X},\rho)$ is complete.}

The following important fact can be found in   \cite{HS, S3}: for a Dirichlet space with
  the above formulated assumptions  
   \textit{existence of a local Poincar\'e inequality is equivalent to existence of  short time Gaussian estimates for the corresponding heat kernel.}

\section{Examples}\label{examples}

\subsection{Compact Riemannian manifolds}

One can easily check that  in the case of $\mathbb{R}^{n}$ or torus $\mathbb{T}^{n}$ if $\mathcal{L}$ is the regular Laplace operator $ \>\Delta =\sum_{j}\frac{\partial^{2} }{\partial x_{j}^{2}}\>$ then for smooth $\>f$
$$
\Gamma(f,f)=|\nabla f|^{2}=\sum_{j} \left |\frac{\partial f}{\partial x_{j}}\right|^{2}
$$
and the local Poincare inequality takes the form
$$
\int_{B(x,\rho)}|f(y)-f_{B}|^{2}d\mu(y)\leq C\rho^{2}\int_{B(x,\rho)}|\nabla f(y)|^{2}d\mu(y).
$$
In these cases  the equivalence
$$
\left\| \mathcal{L}^{1/2}f\right\|^{2}\sim \int_{\mathbf{X}}|\nabla f(x)|^{2}d\mu(x).
$$
holds true.

A similar  Poincare inequality holds on a compact Riemannian manifold ${\bf M}$ if $\mathcal{L}$ is the corresponding Laplace-Beltrami operator. 
All the  properties formulated above  are known to hold  on compact Riemannian manifolds.
The inequality  (\ref{grad}) (and even equivalence)  is known for compact Riemannian manifolds.
Moreover, it is  known that in the case of a compact Riemannian manifold of dimension $d$ the estimate (\ref{Gauss}) holds and takes the following form (due to the fact that in Riemannian situation  all the  balls of radius $\sqrt{t}$ have essentially the same volume $t^{d/2},\>\>d=dim\>{\bf M}$)
\begin{equation}\label{heatkern}
 C_{1} t^{-d/2}e^{-c_{1}\frac{ \left(dist(x,y)\right)^{2}}{t}}\leq P_{t}(x,y)\leq 
 C_{2} t^{-d/2}e^{-c_{2}\frac{ \left(dist(x,y)\right)^{2}}{t}}
\end{equation}
where $0<t<1$ and every constant depends on ${\bf M}$.

\subsection{Compact homogeneous sub-Riemannian manifolds}

Another interesting set of examples is given by compact homogeneous sub-Riemannian manifolds.

Let     $\mathbf{M}=\mathbf{G}/\mathbf{H}$ where $ \mathbf{G, H}$ are a compact Lie groups be a compact homogeneous manifold. Let ${\bf X}=\{X_{1},\ ...,X_{d}\}$ be a basis of the Lie algebra $\mathbf{g}$ of the group $\mathbf{G}$ . Let 
\begin{equation}\label{hvf}
{\bf Y}=\{Y_{1},...,Y_{m}\}
\end{equation}
 be a subset of ${\bf X}=\{X_{1},\ ...,X_{d}\}$ such that $Y_{1},...,Y_{m}$ and all their commutators 
\begin{equation}\label{com}
Y_{j,k}=[Y_{j}, \>Y_{k}]=Y_{j}Y_{k}-Y_{k}Y_{j},\>\>\>
$$
$$
Y_{j_{1},...,j_{n}}=[Y_{j_{1}},
[....[Y_{j_{n-1}}, Y_{j_{n}}]...]],
\end{equation}
 of order $n\leq \mathcal{E}$ span the entire algebra $\mathbf{g}$.

 Let 
\begin{equation}\label{vf}
Z_{1}=Y_{1},  Z_{2}=Y_{2},  ... , Z_{m}=Y_{m}, \>\>\>... \>\>\>, Z_{N},
\end{equation}
be an enumeration of all commutators (\ref{com}) up to order $n\leq \mathcal{E}$. If a $Z_{j}$ corresponds to a commutator of length $n$ we say that $deg(Z_{j})=n$.
Images of vector fields (\ref{vf}) under the natural projection $p: \mathbf{G}\rightarrow \mathbf{M}=\mathbf{G}/\mathbf{H}$ span the tangent space to $\mathbf{M}$ at every point and will be denoted by the same letters. 
\begin{defn}
A sub-Riemann structure on $\mathbf{M}=\mathbf{G}/\mathbf{H}$ is defined as a set of vectors fields on $\mathbf{M}$ which are images of the vector fields (\ref{hvf}) under the projection $p$. They can also be identified with differential operators in $L_{p}({\bf M}),\>1\leq p<\infty,$ under the quasi-regular representation of ${\bf G}$.
\end{defn}
One can define  a non-isotropic metric $\pi$ on ${\bf M}$ associated with the fields $\{Y_{1},...,Y_{m}\}$ (see \cite{NSW}).
\begin{defn}
The following inequalities describe relations of this metric with an $\mathbf{G}$-invariant Riemannian metric $dist$ on ${\bf M}=\mathbf{G}/\mathbf{H}$:
$$
a \>dist(x,y)\leq \pi(x,y)\leq b\left(dist(x,y)\right)^{1/\mathcal{E}}
$$
for some positive $a=a(Y_{1},...,Y_{m}),\> b=b(Y_{1},...,Y_{m})$.
Suppose that  $C(\epsilon)$ denotes the class of absolutely continuous mappings $\varphi: [0,1]\rightarrow {\bf M}$ which almost everywhere satisfy the differential equation 
$$
\varphi^{'}(t)=\sum_{j=1}^{m}h_{j}(t)Z_{j}(\varphi(t)),
$$
where $|h_{j}(t)|<\epsilon^{deg(Z_{j})}$. Then we define 
$
\pi(x,y)$ as the lower bound of all such $\epsilon>0$ for which there exists $\varphi \in C(\epsilon)$ with $\varphi(0)=x,\> \varphi(1)=y$.
\end{defn}
The corresponding family of balls in ${\bf M}$ is given by 
$$
B^{\pi}(x,\epsilon)=\{y\in {\bf M} : \ \pi(x,y)<\epsilon\}.
$$
These balls reflect the non-isotropic nature of the vector fields $Y_{1},...,Y_{m}$ and their commutators.  For a small $\epsilon>0$  ball $B^{\pi}(x,\epsilon)$ is of size $\epsilon$ in the directions $Y_{1},...,Y_{m}$, but only of size $\epsilon^{n}$  in the directions of commutators of length $n$. As a result, \textit{balls of the same radius may have different volumes}. 
 Note that in the case of a Riemannian manifold of dimension $n$ all the balls $B(x,\epsilon)$ of the same radius have essentially the same volume of order $\epsilon^{n}$.
 
We will be interested in  the following sub-elliptic operator (sub-Laplacian)
\begin{equation}\label{sub-L}
-\mathcal{L}=Y_{1}^{2}+...+Y_{m}^{2}
\end{equation}
which is hypoelliptic  \cite{Hor}, self-adjoint and non-negative in $L_{2}(\bf {M})$.
One can easily check that the corresponding gradient which was introduces above is given by the formula
$$
|\Gamma(f,f)|=|Y_{1}f|^{2}+...+|Y_{m}f|^{2}.
$$
In this setting the following Poincare inequality  is known: there exists a $C>0$ such that for any $
B^{\pi}(x,\epsilon)$ and any $f\in C^{\infty}\left(  \overline{ B^{\pi}(x,\epsilon) }   \right)$
\begin{equation}\label{Poinc}
\int_{B^{\pi}(x,\epsilon) }|f-f_{B^{\pi}(x,\epsilon) }|^{2}\leq C\epsilon^{2}\int_{B^{\pi}(x,\epsilon) }\left|\Gamma(f,f)\right|^{2},
\end{equation}
where
$$
f_{B^{\pi}(x,\epsilon) }=|B^{\pi}(x,\epsilon)|^{-1}\int_{B^{\pi}(x,\epsilon) }f.
$$
The Gaussian estimates for the corresponding heat kernel are also known \cite{M}.

\subsection{Sphere $S^{2}$ with a sub-Riemannian metric. A sub-Laplacian and sub-elliptic spaces on $S^{2}$}

To illustrate nature of sub-Riemannian manifolds let's consider the case of the two-dimensional sphere ${\bf S}^{2}$ in the space$(x_{1}, x_{2}, x_{3})$.  We consider on ${\bf S}^{2}$ two vector fields $Y_{1}=X_{2,3}$ and $Y_{2}=X_{1,3}$ where $X_{i,j}=x_{i}\partial_{j}-x_{j}\partial_{i}$.  
The corresponding sub-Laplace operator is
$$
-\mathcal{L}=Y_{1}^{2}+Y_{2}^{2},
$$
and the corresponding "gradient squared"   is
$$
\Gamma(f,f)=\left|Y_{1}f\right|^{2}+\left|Y_{2}f\right|^{2}.
$$
All the previous assumptions are satisfied in this case. Note that since  the operators $Y_{1},\> Y_{2}$  do not span the tangent space to ${\bf S}^{2}$ along a great circle with $x_{3}=0$ the operator $\mathcal{L}$ is not elliptic on $\mathbf{S}^{2}$. Since $Y_{1},\>Y_{2}, $ and their commutator $Y_{3}=Y_{1}Y_{2}-Y_{2}Y_{1}=X_{1, 2}$ span the tangent space at every point of ${\bf S}^{2}$, balls of a small radius $\epsilon$ with center on the circle $(x_{1}, x_{2}, 0)$ have volumes of order $\epsilon^{3}$ but balls with centers away from this circle have volumes of order $\epsilon^{2}$. 
Note \cite{Hor},  that the operator  $
-\mathcal{L}
$ is hypoelliptic.

\section{Appendix. Proof of Lemma \ref{key}}\label{proof}

In this proof we closely follow \cite{CKP}. 
Note, that the Gaussian estimates (\ref{Gauss})  clearly imply that for every $x\in {\bf X}$  and every $t>0$ the quantity $1/P_{t} (x,x)$ is comparable to the volume of the ball $B(x, t^{1/2})$. Namely,
\begin{equation}\label{double-diag}
\frac{C_{1}}{|B(x, t^{1/2})|}\leq P_{t}(x,x)\leq \frac{C_{2}}{|B(x, t^{1/2})|},\>\>\>|B(x, t^{1/2})|=\mu(B(x, t^{1/2})).
\end{equation}
For any measurable  bounded function $F(\lambda),\>\>\lambda\in [0, \>\infty)$ \textit{with compact support } and any $t>0$ one defines a bounded operator $F(t\sqrt{L})$ by the formula
\begin{equation}\label{func-calc-2}
F(t\sqrt{L})f(x)=\int_{\mathbf{X}}K^{F}_{t}(x,y)f(y)dy=\left<K^{F}_{t}(x,\cdot),f(\cdot)\right>,
\end{equation}
where $f\in L_{2}(\mathbf{X})$ and 
\begin{equation}\label{KERN}
K^{F}_{t}(x,y)=\sum_{l=0}^{\infty} F(t\sqrt{\lambda_l})\psi_l(x)\overline{\psi_l(y)} = K^{F}_t(y,x). 
\end{equation}
The function $K^{F}_t$ is known as the kernel of  the operator $F(t\sqrt{L})$.
Let us notice, that if  $0\leq F_{1}\leq F_{2}$ and both of them are bounded and have compact supports then $~K^{F_{1}}_{t}(x,x)\leq K^{F_{2}}_{t}(x,x)~$ for any $x\in {\bf X}$ and $t>0$. 
Indeed, in this case $F_{2}=F_{1}+H$, where $H$ is not negative. By (\ref{KERN}) we have 
\begin{equation}\label{kern-11}
K^{F_{2}}_{t}(x,x)=K^{F_{1}}_{t}(x,x)+K^{H}_{t}(x,x)
\end{equation}
where each term is non-negative. 
Next, we note that using the right-hand side of  (\ref{double-diag}), property (\ref{kern-11}), and the inequality
$
\chi_{[0,\>s]}(\lambda)\leq ee^{-s^{-2}\lambda^{2}}
$ 
we obtain
\begin{equation}\label{kern-2}
\sum_{l,\>\lambda_{l}\leq s}|\psi_{l}(x)|^{2}\leq   e\sum_{l,\>\lambda_{l}\leq s}  e^{-s^{-2}\lambda_{l}^{2}}|\psi_{l}(x)|^{2}\leq 
$$
$$
e\sum_{l\in \mathbb{N}}  e^{-s^{-2}\lambda_{l}^{2}}|\psi_{l}(x)|^{2}=eP_{s^{-2}}(x,x)\leq \frac{a_{2}}{|B(x, s^{-1})|},\>\>\>\>a_{2}=a_{2}({\bf X}).
\end{equation}
To prove the left-had side of   (\ref{double}) consider the inequality 
$$
e^{-t\lambda^{2}}=e^{-t\lambda^{2}}\chi_{[0, \>s]}+\sum_{j\geq 0}\chi_{[2^{j}s,\>2^{j+1}s]}(\lambda)e^{-t\lambda^{2}}\leq 
$$
$$
\chi_{[0, \>s]}+\sum_{j\geq 0}\chi_{[0,\>2^{j+1}s]}(\lambda)e^{ -t2^{2j}s^{2}  },
$$
which implies 
\begin{equation}\label{b-h-k}
 P_{t}(x,x)\leq K^{\chi_{[0,  s]}}_{1}(x,x)+\sum_{j>0}e^{ -t2^{2j}s^{2}  }K^{j}_{1}(x,x),
\end{equation}
where  
$$
K^{\chi_{[0,  s]}}_{1}(x, x)=\sum_{l,\>\lambda_{l}\leq s}|\psi_{l}(x)|^{2},
$$
 $K^{\chi_{[0,  s]}}_{1}(x, y)$ being the kernel of the operator $\chi_{[0,  s]}\left(\sqrt{\mathcal{L}}\right)$ and 
 $$
 K^{j}_{1}(x,x)=\sum_{l,\>\lambda_{l}\leq 2^{j+1}s}|\psi_{l}(x)|^{2},
 $$
  $ K^{j}_{1}(x, y)$       being the kernel of the operator $\chi_{[0,\>2^{j+1}s]}(\sqrt{\mathcal{L}})$.
In conjunction   with (\ref{Gauss}) it gives 

\begin{equation}\label{b-h-k-10}
c_{3}|B(x, t^{-1/2})|\leq P_{t}(x,x)\leq K^{\chi_{[0,  s]}}_{1}(x,x)+\sum_{j>0}e^{ -t2^{2j}s^{2}  }K^{j}_{1}(x,x)=
$$
$$
\sum_{l,\>\lambda_{l}\leq s}|\psi_{l}(x)|^{2}+\sum_{j>0}e^{ -2^{2j}ts^{2}  }\sum_{l,\>\lambda_{l}\leq 2^{j+1}s}|\psi_{l}(x)|^{2}.
\end{equation}
Note, that according to (\ref{doubling})  if $\rho>1$ and $\rho s^{-1}$ is sufficiently small then 
\begin{equation}\label{D-cond}
|B(x, \rho s^{-1}|\leq c\rho^{D}|B(x, s^{-1})|.
\end{equation}
Next, given $s\geq 1$ and $m\in \mathbb{N}$ we pick $t$ such that 
\begin{equation}\label{cond}
s\sqrt{t}=2^{m}.
\end{equation}
The  inequality (\ref{D-cond})  and the condition (\ref{cond}) imply 
\begin{equation}\label{D1}
\frac{(c2^{mD})^{-1}}{|B(x, s^{-1})|}\leq \frac{1}{|B(x, 2^{m}s^{-1})|}\leq c_{1}|B(x, t^{-1/2})|, \>\>\>m\in \mathbb{N}, 
\end{equation}
and
\begin{equation}\label{D2}
\frac{1}{|B(x, 2^{-m-1}s^{-1})|}\leq \frac{c2^{(m+1)D}}{|B(x, s^{-1})|}.
\end{equation}
Thus according to (\ref{D1}),           (\ref{kern-2}) and (\ref{b-h-k-10})                       we obtain that for a certain constant $c_{2}=c_{2}({\bf X})$

\begin{equation}\label{Eq}
\frac{(c2^{mD})^{-1}}  {|B(x, s^{-1})|}\leq c_{1}|B(x, t^{-1/2})|\leq
$$
$$
c_{2}\left( \sum_{l,\>\lambda_{l}\leq s}|\psi_{l}(x)|^{2}+  \sum_{j>0}e^{ -2^{2j}ts^{2}  }\sum_{l,\>\lambda_{l}\leq 2^{j+1}s}|\psi_{l}(x)|^{2}\right)\leq
$$
$$
c_{2}\left( \sum_{l,\>\lambda_{l}\leq s}|\psi_{l}(x)|^{2}+  \sum_{j>0}\frac{ e^{ -2^{2j}ts^{2}  }}{|B(x, 2^{-j-1}s^{-1})|}\right).
\end{equation}
Using  (\ref{Eq}), (\ref{cond}),  and (\ref{D2}) we obtain that for a certain constant $a=a({\bf X})$

$$
\frac{(c2^{mD})^{-1}}  {|B(x, s^{-1})|}\leq
a\left( \sum_{l,\>\lambda_{l}\leq s}|\psi_{l}(x)|^{2}+    \sum_{j>0}\frac {  e^{ -2^{2j}2^{2m } } 2^{(j+1)D} }{|B(x, s^{-1})|}\right)\leq
$$
$$
c_{2}\left( \sum_{l,\>\lambda_{l}\leq s}|\psi_{l}(x)|^{2}+   \frac{2^{D}}  {|B(x, s^{-1})|} \sum_{j>0}  e^{ -2^{2j}2^{2m } } 2^{jD} \right).
$$
Since
$$
 \frac{2^{D}}  {|B(x, s^{-1})|} \sum_{j>0}  e^{ -2^{2j}2^{2m } } 2^{jD}\leq  \frac{2^{D}2^{-mD}}  {|B(x, s^{-1})|} \sum_{j>0}  e^{ -2^{2j }2^{2m } } 2^{(j+m)D}\leq 
 $$
 $$
  \frac{2^{D}2^{-mD}}  {|B(x, s^{-1})|} \sum_{j>0}  e^{ -2^{2(j+m) } } 2^{(j+m)D}\leq   \frac{2^{D}2^{-mD}}  {|B(x,s^{-1})|} \sum_{j>m}  e^{ -2^{2j } } 2^{jD},
 $$
 one has that there are positive constants $c_{3},\>c_{4}$ such that for all  sufficiently large $s$ and $m\in \mathbb{N}$
$$
 \frac{2^{-mD}}  {|B(x, s^{-1})|} 
  \left( c_{3}-c_{4}2^{D}\sum_{j\geq m}  e^{ -2^{2j}2^{j D} } \right)\leq  \sum_{l,\>\lambda_{l}\leq s}|\psi_{l}(x)|^{2},
$$
where expression in parentheses is positive for sufficiently large $m\in \mathbb{N}$.
It proves  the left-had side of   (\ref{double}).
Lemma \ref{key} is proven.

	\end{document}